\documentclass{amsart}

\usepackage[longtable]{multirow}

\usepackage{graphicx}
\usepackage{longtable}

\usepackage{amsmath}
\usepackage{amsthm}
\usepackage{amsfonts}
\usepackage{mathtools}
\usepackage{fullpage}

\usepackage{amssymb}
\theoremstyle{plain}
   \newtheorem{twierdzenie}{Theorem}

\theoremstyle{definition}

   \newtheorem*{example}{Example}

\renewcommand{\arraystretch}{1.2}

\newcommand*{\Z}[0]{{\mathbb Z}}

\usepackage{rotating}
\usepackage{accents}
\title[Period integrals of rigid double octic Calabi-Yau threefold]
{Computing period integrals \\of rigid double octic Calabi-Yau threefolds \\with Picard-Fuchs operator}
\author{Tymoteusz Chmiel}
\address{Jagiellonian University,
ul. {\L}ojasiewicza 6,
30-348 Krak\'ow,
Poland
}
\begin{document}
\maketitle
\vspace{-8mm}
\begin{abstract}
  We present a method for numerical
  computation of period integrals of a rigid Calabi-Yau threefold using
  Picard-Fuchs operator of a one-parameter smoothing. Our method gives a possibility of computing the lattice of period integrals of a rigid double octic without any explicit knowledge of its geometric properties, employing only simple facts from the theory of Fuchsian equations and computations in MAPLE with a library for differential equations. As a surprising consequence we also get approximations of additional integrals related to a singular (nodal) model of the considered Calabi-Yau threefold.
\end{abstract}

\section*{Introduction} 

In this paper by a \emph{double octic} we understand a Calabi-Yau
threefold obtained as a resolution of singularities of a double cover
of $\mathbb{P}^3$ branched along a sum of eight planes in
specific configurations. There exist eleven rigid double octic
Calabi-Yau threefolds defined over $\mathbb Q$; they are
modular with explicitly given modular form (for details see
\cite{Meyer}).

If $X$ is a rigid Calabi-Yau threefold, then for a fixed complex
volume form \linebreak $\omega\in H^{3,0}(X)$ period integrals form a
lattice
 \begin{eqnarray*}
&&\Lambda:=\left\{\int_{\gamma}\omega:\;\gamma\in H_{3}(X,\mathbb{Z})\right\} \subset \mathbb{C}.
\end{eqnarray*}
Cynk and van Straten \cite{periods} used an explicit numerical integration over
three-cycles defined by real polyhedral cells to compute approximations of
period integrals of eleven rigid double octic Calabi-Yau
threefolds. These computations gave numerical evidences that the period
integrals are proportional to the special values of the $L\textnormal{-function}$ for
the corresponding modular form.

We start with the observation that birational models of double octics
appear as singular elements at conifold points of one-parameter
families of manifolds $\{X_t\}_{t\in\mathbb{C}}$, whose general
element is a Calabi-Yau threefold with $h^{1,2}(X)=1$. With each of
these families we can associate a holomorphic solution of certain differential equation of Fuchsian type. 
 
Using basic properties of the monodromy group of a Fuchsian equation,
we were also able to obtain values proportional to the special values of
$L$-functions of respective modular forms by means of solutions' analytic
continuation along a loop starting near the conifold point and encircling
the point of maximal unipotent monodromy. 

Our method gives a possibility of computing the lattice of period integrals of a 
~rigid double octic without any explicit knowledge of its geometric properties, employing only simple facts from the theory of Fuchsian equations and computations in MAPLE with a library for differential equations. As a surprising consequence we also get approximations of additional integrals related to a singular (nodal) model of considered Calabi-Yau threefold.

Recently Ruddat and Siebert proposed a different method of computing period integrals using toric degenerations (see \cite{tropic}).

The first two sections of our paper are of preliminary nature -- we collect necessary information on double octic Calabi-Yau threefolds and Picard-Fuchs operator. Third section contains the description of our method for computing period integrals, and in the fourth section we present the implementation in MAPLE. The last section collects the results of our computation. In the Appendix, for the reader's convenience, we list the Picard-Fuchs operators of one-parameter families of double octic Calabi-Yau threefolds with a MUM point and a conifold point.

This paper is an extended version of my bachelor thesis written at the Jagiellonian University under the supervision of prof. S\l{}awomir Cynk.

\section{Double octic Calabi-Yau threefolds}

In this paper we shall be concerned with Calabi-Yau threefolds
constructed as crepant resolutions of double covers of the
projective space $\mathbb{P}^{3}$ branched along an arrangement of eight planes
$S=P_{1}\cup\dots\cup P_{8}$. 
If the arrangement has no sixfold point and no fourfold line, the double cover admits a~projective crepant resolution of
singularities; we call the resulting Calabi-Yau threefold  a~\emph{double
  octic}. 

A Calabi-Yau threefold $X$ is called \emph{rigid} if it admits no infinitesimal
deformations or equivalently if $h^{1,2}(X)=0$. C. Meyer in \cite{Meyer}
gave a list of eleven rigid double octic Calabi-Yau threefolds defined
over $\mathbb Q$ and 63 examples of one-parameter families; the numbering of  arrangements in our paper follows the one given there.

In tab.~1
we give the equations of eleven arrangements of eight planes which result in a~rigid resolution. For details on
double octics we refer the reader to the monograph~\cite{Meyer}. 

\begin{table}[h]
  \centering
  \[
	\begin{array}{|c|c|}
	\hline\rule[-0.5mm]{0mm}{4mm}
	\textup{Arrangement}&\textup{Equation} \\
	\hline
	1&xyzu(x + y)(y + z)(z + u)(u + x)\\
	\hline
	3&xyzu(x + y)(y + z)(y -u)(x - y -z + u)\\
	\hline
	19&xyzu(x + y)(y + z)(x - z - u)(x + y + z - u)\\
	\hline
	32&xyzu(x + y)(y + z)(x - y - z - u)(x + y - z + u)\\
	\hline
	69&xyzu(x + y)(x - y + z)(x - y - u)(x + y - z - u)\\
	\hline
	93&xyzu(x + y)(x - y + z)(y - z - u)(x + z - u)\\
	\hline
	238&xyzu(x + y + z - u)(x + y -z + u)(x - y + z + u)(-x + y + z + u)\\
	\hline
	239&xyzu(x + y + z)(x + y + u)(x + z + u)(y + z + u)\\
	\hline
	240&xyzu(x + y + z)(x + y - z + u)(x - y + z + u)(x - y - z - u)\\
	\hline
	241&xyzu(x + y + z + u)(x + y - z - u)(y - z + u)(x + z - u)\\
	\hline
	245&xyzu(x + y + z)(y + z + u)(x - y - u)(x - y + z + u)\\
	\hline
	\end{array}
	\] 
  \caption{Rigid double octic Calabi-Yau threefolds}
\end{table}
	
By the modularity theorem (cf. \cite{modular}) every rigid Calabi-Yau
threefold defined over $\mathbb Q$ is modular, i.e. its $L$-series is equal
to an $L$-series  $L(f,s)$ of a modular form  $f$ of weight four for
some $\Gamma_{0}(N)$. In the table below we give the
modular forms corresponding to rigid arrangements (we again adopt the notation from \cite{Meyer}):
\begin{table}[h]
\centering	
\[
  \begin{array}{|r|l|l|}
    \hline
    \textup{Form}&\textup{q-series expansion}&\textup{Arrangements}\\
    \hline
    6/1&q - 2q^2 - 3q^3 + 4q^4 + 6q^5 + 6q^6 - 16q^7 + O(q^{8})    &240, 245\\
    8/1&q - 4q^3 - 2q^5 + 24q^7 - 11q^9 - 44q^{11} + O(q^{12})   &1, 32, 69, 93, 238, 241\\
    12/1&q + 3q^3 - 18q^5 + 8q^7 + 9q^9 + 36q^{11} + O(q^{12}) &239\\
    32/1&q + 22q^5 - 27q^9 + O(q^{12})   &19\\
    32/2&q + 8q^3 - 10q^5 + 16q^7 + 37q^9 - 40q^{11} + O(q^{12})&3\\
    \hline
  \end{array}
\]
\caption{Modular forms for rigid double octics}
\end{table}

For our purposes it is crucial that every rigid double octic is birational to a~special member $X_{t_{0}}$ of a one-parameter family $X_{t}$; all special elements of one-parameter
families of double octics are given in \cite{CynkKocel}.  Taking the
resolution of the general member of a family of branched double covers, we get a
partial resolution of the special element corresponding to a rigid arrangement. This gives a special form of \emph{geometric transition}.

\section{Picard-Fuchs operator of a one-parameter family}
 Let $X$ be a Calabi-Yau manifold with $h^{1,2}(X)=1$. By the
 Bogomolov-Tian-Todorov theorem $X$ has a one-dimensional deformation space 
 $\{X_t\}_{t\in V}$, where $V$ is a neighbourhood of $0$ in $\mathbb
 C$ and $X_0=X$. The family $\{X_t\}_{t\in V}$ is locally trivial, i.e. it
 is locally diffeomorphic to $X_0\times \Delta$.
 
 Let $\omega_t$ be the complex volume form on $X_t$, depending
 holomorphically on $t$. If we fix a 3-cycle \linebreak $\gamma_{0}\in
 H_3(X_{0},\mathbb{Z})$ on $X_{0}$, by local triviality we can
 extend it to $\gamma_t \in H_3(X_{t},\mathbb{Z})$ for all $t\in V$. The function
 $y(t):= \int_{\gamma_t}\omega_t$ is called the
 \textit{period function}.
 
Period function of a one-parameter family of Calabi-Yau threefolds
over a complement of a finite set in $\mathbb {P}^{1}$
satisfies a differential equation, called \emph{the Picard-Fuchs equation} of this family. It is a linear
differential equation of order four
 \begin{equation} \label{eq:1}
   y^{(4)}+p_{1}y^{(3)}+p_{2}y''+p_{3}y'+p_{4}y=0,
 \end{equation}
with rational coefficients $p_i \in
\mathbb{C}(X)$. We can also write the Picard-Fuchs operator in the form 
\begin{equation} \label{eq:2}
  \Theta^4+q_{1}\Theta^{3}+q_{2}\Theta^2+q_{3}\Theta+q_4=0, \end{equation} 
where $\Theta=t\cdot\frac{d}{dt}$ denotes the logarithmic derivation, $q_i \in
\mathbb{C}(X)$.

\smallskip

The Picard-Fuchs operator of a one-parameter family of Calabi-Yau
threefolds is \emph{Fuchsian}, i.e. it has only regular singular
points. Equivalently the functions $q_{i}$ are in fact polynomials. More detailed information on Picard-Fuchs operators of one-parameter
families of Calabi-Yau threefolds can be found in \cite{Str}.
 
The local exponents of the Picard-Fuchs operator at a singular point $t_0$
corresponding to a singular Calabi-Yau threefold are equal to
$\{\alpha-\delta,\alpha,\alpha,\alpha+\delta\}$ with rational
$\alpha,\delta$,  $\delta>0$.
Such a regular singular point $t_0$ is called a
\emph{conifold point} (or a \emph{C point}). After a pullback we can assume the
local exponents to be $\{0,1,1,2\}$ and then there is a
fundamental system of solutions of the Picard-Fuchs equation of the form $$f_1,f_2,f_3+\log(t-t_{0})f_2,f_4,$$ where
 for $i=1,\dots,4$ the functions $f_i$ are holomorphic in the neighbourhood of the
 conifold point and $f_1, f_{2}, f_{3}$, and $f_{4}$ have orders
 0, 1, 1, and 2 respectively. 
 
 Our computation depends also on the existence of a point of
 \textit{maximal unipotent monodromy} (MUM) -- regular singular
 point $t_0$ with local exponents $\{0,0,0,0\}$. At a~MUM point we can
 choose a fundamental system of solutions of the 
 form $$f_1,f_2+\log(t-t_0)f_1,f_3+\log(t-t_0)f_2+\frac{1}{2}(\log(t-t_0))^2f_1,$$
 $$f_4
 +\log(t-t_0)f_3+\frac{1}{2}(\log(t-t_0))^2f_2+\frac{1}{6}(\log(t-t_0))^3f_1,$$ 
 where for $i=1,\dots,4$ the functions $f_i$ are holomorphic in its neighbourhood
 and $f_1$ has a non-zero free term.
  
  We shall also use notation \textit{MUM+a point}, resp. \textit{C+a point},
 $a\in\mathbb{Q}$, for a point with local exponents $\{a,a,a,a\}$,
 resp. $\{a,a+1,a+1,a+2\}$, and an appropriate fundamental system of
 solutions. Similarly,  $aC$ \textit{point} denotes a point
 with local exponents $\{0,a,a,2a\}$. 

 A solution $f$ of the Picard-Fuchs operator at point $t_{0}$ can be
 continued along  any path $\gamma : [0,1]\longrightarrow \mathbb {P}^{1}\setminus\Sigma$
 with $\gamma(0)=t_{0}$, omitting the set of singular points $\Sigma$, and
 produces a solution $T_{\gamma}(f)$ at the point $\gamma(1)$. Moreover, the
 continuations $T_{\gamma}(f)$ along homotopic paths are equal.
 
 Continuation along a loop  based at any point $t\not\in\Sigma$ gives a linear
transformation of the space of solutions in a neighbourhood of this
point. If $t_{0}\in\Sigma$ is a singular point and we take a small
disc $\Delta$ centered at $t_{0}$ that does not contain any other
singularities, the continuation along the boundary of $\Delta$ defines the
local monodromy operator $T_{t_{0}}$.

 \section{Description of the method}
 
 We shall consider a family $X_{t}$ of double octic Calabi-Yau
 threefolds such that a~general member of this family satisfies
 $h^{1,2}(X_t)=1$, $t_{0}$ is
 a conifold point and there is a MUM point, which we may assume to be~$0$. Moreover assume that $X_{t_{0}}$ is a~singular variety that
 resolves to a~rigid Calabi-Yau threefold $\overline{X}_{t_0}$.
 
 The computations in \cite{periods} yield numerical evidences that
 real and imaginary period integrals of $\overline{X}_{t_0}$ are integral multiples of
 $\pi^{2}L(f,1)$ and $\pi L(f,2)$, where $L(f,1)$ and $L(f,2)$ are the
 special values of the $L$-function of the modular form corresponding to
 $\overline{X}_{t_0}$.

 We shall present a method which in principle allows one to compute an arbitrarily precise numerical approximation of periods of $\overline{X}_{t_0}$. In our computations
 we shall not use any geometric properties of the one-parameter family, in
 particular we shall not use the fact that $X_{t}$ is a double
 octic. In fact, we will only use the Picard-Fuchs operator of the
 family $X_{t}$ computed in \cite{prep}. Due to this fact, we have refrained from listing the equations of one-parameter families and in the Appendix we present only their associated Picard-Fuchs operators.
 
 Let us recall that in our situation the element of the family of Calabi-Yau
 threefolds $X_{t}$ is defined by a resolution of a double covering of
 $\mathbb{P}^{3}$ branched along an arrangement of eight planes
 $S_{t}=P_{1}^{t}\cup\dots\cup P^{t}_{8}$.
 Special elements $X_{t_{0}}$ of this family correspond
 to a choice of four indices $1\le i_{1}<i_{2}<i_{3}<i_{4}\le 8$ such
 that the planes $P_{i_{1}}^{t},P_{i_{2}}^{t},P_{i_{3}}^{t},P_{i_{4}}^{t}$ are in
 general position for generic  $t$, but
 $P_{i_{1}}^{t_{0}},P_{i_{2}}^{t_{0}},
 P_{i_{3}}^{t_{0}},P_{i_{4}}^{t_{0}}$ intersect.
 The resolution of singularities of a generic element of a family of double covers yields also a partial resolution of the double cover at $t_{0}$.
 
 Assume that the intersection of the planes $P_{i_{1}}^{t_{0}},P_{i_{2}}^{t_{0}},
 P_{i_{3}}^{t_{0}},P_{i_{4}}^{t_{0}}$ is a fourfold point of
 $S_{t_{0}}$ which is not contained in a triple line of $S_{t_{0}}$
 (we call such a singular point $p_{4}^{0}$). Then the degeneration of
 $S_{t}$ at $t_{0}$ is given by a~vanishing tetrahedron, which gives a
 vanishing cycle in $X_{t}$. The singular element $X_{t_{0}}$ has two
 ordinary double points (nodes) as its only singularities.
 
 If the special fiber is nodal, gluing a 4-cell along the vanishing
 cycle we get topological space homotopic to the special fiber. The
 same holds true for a small resolution with a 3-cell glued along the
 exceptional line.
 \begin{twierdzenie}[\mbox{\cite[Ch. II]{Werner}}]
 	Let $X_{t}, t\in\Delta$, be a family of projective varieties such
 	that
 	\begin{itemize}
 		\item $X_{t}$ is a (smooth) Calabi-Yau threefold with
 		$h^{1,2}(X_{t})=1$, for $t\not=0$,
 		\item $X_{0}$ is a nodal variety such that a small resolution
 		$\overline X_{0}$ is a rigid Calabi-Yau threefold. 
 	\end{itemize}
 	Then
 	$H_{3}(X_{0})\cong\mathbb Z^{3}\oplus \textit{torsions}$.
 \end{twierdzenie} 
 
 As an integral over a torsion cycle vanishes, we will only consider homology groups modulo torsions. Denoting by
 $\overline {X}_{t_{0}}$ a (projective) small resolution of $X_{t_{0}}$ we get
 \[H_{3}(X_{t},\Z)\cong \Z^{4}, \quad H_{3}(X_{t_{0}},\Z)\cong \Z^{3},\quad
 H_{3}(\overline{X}_{t_{0}},\Z)\cong \Z^{2}.\] 
 The group $H_{3}(X_{t},\Z)$ is spanned by $H_{3}(X_{t_{0}},\Z)$ and
 the vanishing cycle. From considerations in \cite{periods} it follows that the map
 \[H_{3}(\overline{X}_{t_{0}},\Z)\longrightarrow H_{3}(X_{t_{0}}\Z)\]
 is injective. Moreover, a polyhedral 3-cycle belongs to the image of
 this map iff it is locally symmetric at the $p_{4}^{0}$ point.
 Consequently, the integrals of the 3-form
 \[\omega_{t_{0}}:=\frac{dx\wedge\ dy\wedge dz}{\sqrt{F_{t_{0}}(x,y,z)}},\]
 where $F_{t_{0}}(x,y,z)$ is an equation of the octic arrangement $S_{t_{0}}$
 in an affine chart, over polyhedral 3-cycles in $X_{t_{0}}$ generate a
 group of rank at most 3 in $\mathbb C$, while  integrals over polyhedral 3-cycles
 in $H_{3}(\overline{X}_{t_{0}})$ generate~a rank two subgroup
 (commensurable with the full periods lattice of $\overline
 {X}_{t_{0}}$).
 
 Since we will consider only those integrals 
 the computations in \cite{periods} take into account, 
 in the table below we list the generators of integrals over polyhedrals 3-cycles on a~singular double cover of a rigid arrangement (since we impose no symmetry condition on the cycles at the $p^0_4$ point, we can get more than three independent integrals): 
 
 \bigskip
 \begin{center}
 \begin{tabular}{r|rr|rr}
 	Arr. &\multicolumn{2}{l|}{Real integrals}&
 	\multicolumn{2}{l}{Imaginary integrals}\\\hline
 	3&14.303841078& 18.695683053&
 	41.413458745i\\
 	19&12.3280533145& 19.3301891966&
 	12.3280533145i& 19.3301891966i\\
 	32&11.13352966& 16.85672240&
 	17.34237466i\\
 	69&11.13352966& 16.85672240&
 	17.34237465i\\
 	93&8.42836120319& 11.1335296603&
 	17.3423746625i\\
 	239&13.1823084825& 17.6714531944&
 	11.7425210928i\\
 	240&3.99263311132& 6.94406875218&
 	4.80390756451i& 6.9176905115i\\
 	245&3.99263311132& 6.94406875217&
 	5.38024923409i& 7.49403218155i
 	\\	
 \end{tabular}
 \end{center}
 \bigskip
 
 The integrals over polyhedral 3-cycles in $H_{3}({X}_{t_{0}})$ form a subgroup of rank at most 3 in the group generated by the entries of the above table corresponding to $\overline{X}_{t_0}$. Note that the group $H_{3}({X}_{t_{0}})$ depends not only on the rigid arrangement but also on the choice of the one-parameter family.

 Now recall that at the conifold  point $t_{0}$ we
 have a fundamental system of solutions of the
 form $$f_1,f_2,f_3+\log(t-t_{0})f_2,f_4.$$ 
 Up to multiplication by a constant, two elements of this basis are uniquely determined: $f_4$ as the unique element
 of order 2 and $f_2$ as the coefficient of the logarithmic
 term.
 
 For any solution $f$ of the differential equation near a
 conifold point $t_{0}$, the local monodromy around $t_{0}$ is of the form
 $T_{t_0}(f)=f+c\cdot f_{2}$, where $c$ is a constant.
 The solution $f_{2}$ is therefore proportional to
 $N_{t_{0}}f$, where
 $N_{t_{0}}=T_{t_{0}}-\operatorname{Id}$.   
 In our case where $X_{t_{0}}$ is a nodal variety, $f_{2}(t)$ equals the
 integral over the vanishing cycle on $X_t$.
 
 Now for a differential operator $\mathcal{P}$ with a MUM point at $0$ and a conifold point $t_0$ denote by $\mathcal{L}_{\mathcal{P},t_0}$ the additive group generated by $\{\operatorname{Re}(T_0^n(f_2)),\operatorname{Im}(T_0^n(f_2)): n\in\mathbb{N}\}$. Observe that in our case this will be a~subgroup of integrals over cycles in $H_3(X_{t_0})$ and therefore will have rank at most 3.
 
 \smallskip
 
 Assume that we have two different Picard-Fuchs operators $\mathcal{P}$ and $\mathcal{P}'$ with conifold points $t_0$ and $t_0'$ such that $\overline{X}_{t_0}\cong\overline{X}_{t_0'}$. It follows from the considerations above that if $\mathcal{L}_{\mathcal{P},t_0}$ and $\mathcal{L}_{\mathcal{P'},t_0'}$ are of rank 3 and for some $\alpha\in\mathbb{C}$ the intersection $\mathcal{L}_{\mathcal{P},t_0}\cap\alpha\mathcal{L}_{\mathcal{P'},t_0'}$ is a lattice, this lattice is commensurable with the lattice of period integrals of $\overline{X}_{t_0}$.
 
 On the other hand, in all considered cases if $\mathcal{L}_{\mathcal{P},t_0}$ already is a lattice, it is (numerically) commensurable with the lattice of period integrals.
 
 \section{Computing the periods}
 
 We now describe the implementation of the above general idea. Using MAPLE's procedure \texttt{formal\_sol} we are able to
 determine the solution $f_2$ in a neighbourhood of the conifold point $t_0$. We then  choose a~polygon chain $L=(a_0,\dots,a_{n+1})$ such that
 $a_0=a_{n+1}=t_0+\epsilon$, $|\epsilon|\ll1$, and the winding number of $L$ around every singular point of
 the equation -- except for the  MUM point at $0$ -- equals 0 and the
 winding number around $0$ equals 1. 
 
 Let $g_0,g_1,g_2,g_3$ be a fundamental system of solutions at the point $a_{i+1}$,
 again computed with the MAPLE's \texttt{formal\_sol} procedure. Note that in
 every point of the curve $L$ we have a fundamental system of
 solutions of orders 0, 1, 2, 3. Therefore given a
 solution $g$ in a neighbourhood of $a_i$ whose disc of convergence
 has non-empty intersection with the discs of convergence of all
 solutions at $a_{i+1}$, we can fix a point $c_{i}$ in this intersection
 and consider the system of equations 
 
 $$\begin{bmatrix}
 g(c_i)\\g'(c_i)\\g''(c_i)\\g'''(c_i)
 \end{bmatrix}=\begin{bmatrix}
 X_1g_0(c_i)+X_2g_1(c_i)+X_3g_2(c_i)+X_4g_3(c_i)\\
 X_1g_0'(c_i)+X_2g_1'(c_i)+X_3g_2'(c_i)+X_4g_3'(c_i)\\
 X_1g_0''(c_i)+X_2g_1(c_i)+X_3g_2''(c_i)+X_4g_3''(c_i)\\
 X_1g_0'''(c_i)+X_2g_1'''(c_i)+X_3g_2'''(c_i)+X_4g_3'''(c_i)\\
 \end{bmatrix}.$$
 This gives us representation of $g$ in the basis $g_0,g_1,g_2,g_3$
 and we can continue it analytically to the point $a_{i+1}$
 (cf. \cite{vEvS}).
 
 In our application of this method, we begin with the solution $f_2$ at a point close to the conifold point $t_0$. Repeating the aforementioned computation for all vertices of $L$,
 we continue the solution $f_2$ around the point $0$ and obtain a new
 solution $\tilde{f_2}$ at $t_0+\epsilon$. In our computations we always took
 $c_i=\frac{1}{2}(a_i+a_{i+1})$, therefore while choosing vertices of
 the polyline it was important to check that $c_i$ is sufficiently
 close to the centers of the discs of convergence at $a_i$ and
 $a_{i+1}$ to avoid potential problems with the rate of convergence.
 
 When this procedure is completed, we end up with a new solution
 $\tilde{f_2}:=T_0(f_2)$. As our goal is to determine $\mathcal{L}_{\mathcal{P},t_0}$, we repeat the procedure now starting with $\tilde{f_2}$ instead of $f_2$, getting a new solution $\tilde{\tilde{f_2}}$ and so on. Then we use the following:

 \begin{twierdzenie}\label{recursion}
 	Let $W_{n}:=T_0^n(f_2)(t_0)$, $n\in\mathbb N$. The sequence
 	satisfies the following recursion
 	\[W_{n+4}=4W_{n+3}-6W_{n+2}+4W_{n+1}-W_{n}.\]
 \end{twierdzenie}
 \begin{proof}
 	The monodromy operator $T_{0}$ around a MUM point has a single
 	Jordan block with eigenvalue $1$:
 	\[\left(
 	\begin{array}{rrrr}
 	1&1&0&0\\0&1&1&0\\0&0&1&1\\0&0&0&1
 	\end{array}\right).
 	\]
 	Consequently the characteristic polynomial of $T_{0}$ equals
 	$\chi(T)=(T-1)^{4}$. By the Cayley-Hamilton theorem
 	\[T_{0}^{4}-4T_{0}^{3}+6T_{0}^{2}-4T_{0}+6\operatorname{Id}=0\]
 	and the theorem follows.
 \end{proof}
 
 We therefore see that all possible values of Re$(W_n)$ and Im$(W_n)$ are integral linear combinations of the first four and the generators for those values will generate the entire $\mathcal{L}_{\mathcal{P},t_0}$.
 
 Additional difficulty was posed by examples where instead of MUM and C
 points we encountered points of type MUM$+a$, C$+a$ or $\frac1m$C for some
 $a\in\mathbb{Q}$ and $m\in\mathbb{N}$. In two former cases it is relatively easy to change the local exponents such that we
 get the desired basis: indeed, to change local exponents at $0$ one only needs to represent the differential operator using the logarithmic derivative $\Theta$ and then substitute $\Theta \mapsto \Theta+a$. However normalizing basis at $\frac1m$C points would
 require us to use the substitution $(t-t_0) \mapsto (t-t_0)^{m}$,
 producing new regular singular points in the process and making it
 harder to find the correct polyline. Because of this inconvenience, we have decided to perform the computations for this type of singular points using their natural basis $(t-t_0)^{\frac{1}{m}} f_1,(t-t_0)^{\frac{1}{m}} f_2,(t-t_0)^{\frac{1}{m}}(f_3+\log(t-t_{0})f_2),(t-t_0)^{\frac{1}{m}}f_4$ and take $(t-t_0)^{\frac{1}{m}}f_2$ as the initial solution.
 
 \begin{example}
 We shall now present the details of computations in the simplest case of
 Arrangement 2. It has
 the equation $$(y+z)(x+y)(z+u)(tu+x)xyzu,$$ where $t$ is the
 parameter and $x,y,z,u$ are coordinates in
 $\mathbb{P}^3$. Picard-Fuchs operator for this
 family is $$\mathcal{P}_2:=\Theta^4-\frac{1}{16}t(2\Theta+1)^4.$$ 
 
 The equation $\mathcal{P}_2=0$ has three regular singular points; we present them
 in the following table, called the \textit{Riemann scheme},
 together with their local exponents: \[\left\{\begin{tabular}{*{3}c} 
 0& 1& $\infty$\\ 
 \hline
 0& 0& 1/2\\
 0& 1& 1/2\\
 0& 1& 1/2\\
 0& 2& 1/2\\
 \end{tabular}\right\}\]
 
 Since $1$ is a conifold point, $0$ is a MUM point and the third
 regular singular point is located at infinity, we take
 $$(1,\tfrac{3}{4},\tfrac{1}{2},\tfrac{1}{4},\tfrac{1}{4}+\tfrac{i}{4},\tfrac{i}{4},-\tfrac{1}{4}+\tfrac{i}{4},-\tfrac{1}{4},-\tfrac{1}{4}-\tfrac{i}{4},-\tfrac{i}{4},\tfrac{1}{4}-\tfrac{i}{4},\tfrac{1}{4},\tfrac{1}{2},\tfrac{3}{4},1)$$ as the vertices of the polyline $L$. 
 
 Working with the precision \texttt{Digits:=30} and with the
 parameter \texttt{'order'}$=100$ for the procedure
 \texttt{formal\_sol} from the \texttt{DEtools} library, we continue analytically the solution 
\begin{multline*}$$f_2=(t-1)+\frac{3}{4}(t-1)^2-\frac{29}{48}(t-1)^3+\frac{49}{96}(t-1)^4-\frac{2277}{5120}(t-1)^5+\frac{8107}{20480}(t-1)^6\\-\frac{205223}{573440}(t-1)^7+\frac{37551}{114688}(t-1)^8-\frac{11413801}{37748736}(t-1)^9+\dots$$\end{multline*} 
along $L$.
After the first encircling we got
  $$\tilde{f_2}(1)=-11.3440218793908710004979185926+1.93350327192382796832769889845\cdot10^{-23}i,$$
 after the second --
  $$\accentset{\approx}{f_2}(1)=-45.3760875175634840019915624157+28.1143988476259022087394753013i,$$
and so forth.
  
  It is then checked that the sequences $$\frac{\operatorname{Re}(T_0^n(f_2)(1))}{\operatorname{Re}(T_0(f_2)(1))} \ \ \textnormal{and}\ \  \frac{\operatorname{Im}(T_0^n(f_2)(1))}{\operatorname{Im}(T_0^2(f_2)(1))}$$
 for $n=1,2,3,4$ within precision of 20 digits read $(1,4,9,16)$ and $(0,1,4,10)$. Using Theorem \ref{recursion} we see that in this case $\mathcal{L}_{\mathcal{P}_2,t_0}$ is already a lattice.
 
 The resolution of singularities at the considered conifold point
 corresponds to the double octic of Arrangement 1 with the
 corresponding modular form 8/1. Its special $L$-values are 
 $$L(f,1)\approx 0.35450068373096471876555989149,$$ $$L(f,2)\approx 0.69003116312339752511910542021.$$
 
 We can sum up our computations by
 \begin{eqnarray*}
 &&\frac{\operatorname{Re}(\tilde{f_2}(1))}{L(f,1)}\approx-32.0000000000000000000000056066,\\
 &&\pi\cdot\frac{\operatorname{Im}(\accentset{\approx}{f_2}(1))}{L(f,2)}\approx127.999999999999999999997974854.
 \end{eqnarray*}
 \end{example}

\section{Results of computation}

 In total we have studied 49 conifold points
 appearing in 29 families of double octics.
 
 \smallskip
 
 In 40 cases the described method suggests that $\mathcal{L}_{\mathcal{P},t_0}$ is already a lattice commensurable with the period lattice of the corresponding rigid double octic. Similarly to the example above, in fact we get not only the commensurability of lattices but it also appears that the generators computed by the means of our method are multiples of $L(f,1)$ and $\frac{L(f,2)}{2\pi i}$. In some cases they are multiples of $\sqrt{2}L(f,1)$ and $\sqrt{2}\frac{L(f,2)}{2\pi i}$; the reason is that in order to get the twist of the modular form with the smallest possible level, the equation of the corresponding arrangement of planes was multiplied by 2 (cf. [2]).
 
 The following table presents the generators in those 40 cases:

 \def\lj{L(f,1)}
 \def\ld{\frac{L(f,2)}{2\pi i}}

 \def\arraystretch{1.45}
 \begin{longtable}{|c|c|c|c|c|c|}\hline
   Arr. &\parbox[t]{1.2cm} {confiold\\ point}&\parbox[t]{.8cm}
                                                 {rigid\\ Arr.}&
 \parbox[t]{1.2cm} {modular\\
form}&real period&imaginary period\\[5mm]
   \hline
   2&1&1&8/1&$32L(f,1)$&$256\ld$\\\hline
   8&-1&1&8/1&$64i\ld$&$64i\lj$\\\hline
   10&0&1&8/1&$64i\ld$&$32i\lj$\\\hline
   10&-1&1&8/1&$32\lj$&$64\ld$\\\hline
   16&-1&1&8/1&$32\lj$&$64\ld$\\\hline
   20&0&3&32/2&$4\lj$&$16\ld$\\\hline
   20&-1&1&8/1&$64i\ld$&$96i\lj$\\\hline
   36&-1&32&8/1&$128i\ld$&$128i\lj$\\\hline
   73&$-\tfrac12$&69&8/1&$32i\ld$&$32i\lj$\\\hline
   94&-1&1&8/1&$16\lj$&$128\ld$\\\hline
   95&$\infty$&3&32/2&$16i\ld$&$4i\lj$\\\hline
   99&$-\tfrac12$&19&32/1&$2\sqrt2\lj$&$64\sqrt2\ld$\\\hline
   99&$\infty$&19&32/1&$4\sqrt2\lj$&$32\sqrt2\ld$\\\hline
   144&$-\tfrac12$&19&32/1&$2\sqrt2\lj$&$64\sqrt2\ld$\\\hline
   154&0&1&8/1&$32\lj$&$64\ld$\\\hline
   154&$-\tfrac12$&32&8/1&$16\lj$&$32\ld$\\\hline
   199&0&1&8/1&$64i\ld$&$32i\lj$\\\hline
   199&-1&69&8/1&$64i\ld$&$32i\lj$\\\hline
   242&$-\tfrac12$&238&8/1&$16\lj$&$32\ld$\\\hline
   242&$\infty$&238&8/1&$64\lj$&$128\ld$\\\hline
   246&0&1&8/1&$320i\ld$&$64i\lj$\\\hline
   246&$-\tfrac12$&241&8/1&$256i\ld$&$256i\lj$\\\hline
   249&$-\tfrac12$&241&8/1&$128\lj$&$512\ld$\\\hline
   251&0&1&8/1&$64i\ld$&$64i\lj$\\\hline
   251&$-\tfrac12$&93&8/1&$8\sqrt2i\ld$&$16\sqrt2i\lj$\\\hline
   251&$\infty$&19&32/1&$4\sqrt2\lj$&$16\sqrt2\ld$\\\hline
   253&0&3&32/2&$8\lj$&$16\ld$\\\hline
   253&$-\tfrac12$&240&6/1&$5\lj$&$30\ld$\\\hline
   254&$-\tfrac12$&241&8/1&$8\lj$&$4\ld$\\\hline
   255&2&32&8/1&$32\sqrt2i\ld$&$32\sqrt2\lj$\\\hline
   256&-2&239&12/1&$32\lj$&$96\ld$\\\hline
   256&2&238&8/1&$32\lj$&$256\ld$\\\hline
   257&4&240&6/1&$200\lj$&$288\ld$\\\hline
   259&0&32&8/1&$32\lj$&$128\ld$\\\hline
   259&$\infty$&32&8/1&$32\lj$&$128\ld$\\\hline
   262&0&1&8/1&$320i\ld$&$80\lj$\\\hline
   265&4&69&8/1&$16\lj$&$64\ld$\\\hline
   268&-1/2&69&8/1&$160\sqrt2i\ld$&$32\sqrt2i\lj$\\\hline
   274&0&3&32/2&$32i\ld$&$8i\lj$\\\hline
   274&$\infty$&3&32/2&$32i\ld$&$8i\lj$\\\hline
\caption{Results with $\mathcal{L}_{\mathcal{P},t_0}$ of rank 2}
\label{tab:3}
 \end{longtable}

Notice that, except for the Arrangement 245, each double octic corresponding to a rigid arrangement of eight planes appears as a~resolution of singularities in at least one conifold point listed above. Therefore for those ten double octics described method seems to allow us to compute period integrals using the Picard-Fuchs operator of just one one-parameter family.

Now we list the generators for the remaining nine cases:

 \def\arraystretch{1.35}
 \begin{longtable}{|c|c|c|c|c|}
 \hline  \multirow{2}{*}{Arr.} &\parbox{1.2cm} {special}&\parbox[t]{.8cm}{rigid}&
 \parbox[t]{1.2cm} {modular}&real periods\\
\cline{5-5}   
&point&Arr.&form&imaginary periods\\
   \hline
\hline
   \multirow{3}{*}5&\multirow{3}{*}0&\multirow{3}{*}3&\multirow{3}{*}{32/2}&16.7842426348152009451903095600\\\cline{5-5}
&&&&3.78853747194184773010686231258i\\
        &&&&61.0738884585292464400038239965i\\\hline
    \multirow{3}{*}5&\multirow{3}{*}2&\multirow{3}{*}3&\multirow{3}{*}{32/2}&4.19606065870380023629757253576\\\cline{5-5}
        &&&&0.94713436798546193252671571987i\\
        &&&&3.90285969880676841968001329994i\\\hline
 \multirow{3}{*}{20}&\multirow{3}{*}{-2}&\multirow{3}{*}{19}&\multirow{3}{*}{32/1}&2.58310412634697457527735648888
   \\\cline{5-5}&&&&0.979278824715794481666000593885i\\
   &&&&4.45674355709313141111341743112i\\\hline
\multirow{3}{*}{95}&\multirow{3}{*}{$-\tfrac12$}&\multirow{3}{*}{93}&\multirow{3}{*}{8/1}&6.21246314816860397127669373665\\\cline{5-5}&&&&
2.99683078705084653614316487029i\\
&&&&34.0238543159967756814545903982i \\\hline
\multirow{3}{*}{244}&\multirow{3}{*}{$\tfrac12$}&\multirow{3}{*}{240}&\multirow{3}{*}{6/1}&6.26847094349121003359079636235\\\cline{5-5}&&&&
2.58823590805561845768157001028i\\&&&&32.0498374325403392826453731746i\\\hline
\multirow{3}{*}{244}&\multirow{3}{*}{2}&\multirow{3}{*}{240}&\multirow{3}{*}{6/1}&25.0738837739648401343633473698\\*
\cline{5-5}&&&&10.3529436322224738307260280020i\\*
&&&&128.199349730161357130578977414i\\* \hline
\multirow{3}{*}{253}&\multirow{3}{*}{-2}&\multirow{3}{*}{245}&\multirow{3}{*}{6/1}&8.26021099210331426827789806115\\\cline{5-5}&&&&
6.26847094349121003359079492495i\\&&&&7.96011334055139325749281017005i\\\hline
\multirow{3}{*}{274}&\multirow{3}{*}{$-\tfrac12$}&\multirow{3}{*}{245}&\multirow{3}{*}{6/1}&3.13423547174561288123848517920\\\cline{5-5}&&&&
0.839792675513409448977564062085i\\&&&&8.12249522907253404678014309610i\\\hline
\multirow{3}{*}{274}&\multirow{3}{*}{-2}&\multirow{3}{*}{245}&\multirow{3}{*}{6/1}&5.57197417218209438771001425150\\\cline{5-5}&&&&
1.49296475661811062877124575360i\\&&&&14.4399915143798181025349619989i\\\hline
\caption{Results with $\mathcal{L}_{\mathcal{P},t_0}$ of rank 3}
\label{tab:4}
 \end{longtable}
\vspace*{-8mm}
In each case the real values in $\mathcal{L}_{\mathcal{P},t_0}$ were multiples of the special value, while imaginary parts formed a~group of rank 2.

We see that the remaining case rigid double octic of Arr. 245 appears as a~resolution of a~special fiber in three cases and we get three rank 3 subgroups: $\mathcal{L}_{\mathcal{P}_{253},-2}$, $\mathcal{L}_{\mathcal{P}_{274},-\frac{1}{2}}$ and $\mathcal{L}_{\mathcal{P}_{274},-2}$. The intersection $\mathcal{L}_{\mathcal{P}_{253},-2}\cap i\mathcal{L}_{\mathcal{P}_{274},-\frac{1}{2}}$ has rank 2 and is generated by the multiples of the special values: $144\sqrt{2}i\frac{L(f,2)}{2\pi i}$ and $40\sqrt{2}iL(f,1)$, as to be expected.

This means that for all rigid double octics our method allows us to compute the lattice of period integrals with much higher precision than the methods of numerical integration used in \cite{periods}.

The results in Table \ref{tab:4} can be interpreted by a more detailed
analysis of the computations of period integrals carried out in 
\cite{periods} where the period integrals were computed with numerical
integration over polyhedral cells. Let us recall the table from the Section 3., listing those integrals:

\begin{center}
\begin{tabular}{r|rr|rr}
	Arr. &\multicolumn{2}{l|}{Real integrals}&
	\multicolumn{2}{l}{Imaginary integrals}\\\hline
	3&14.303841078& 18.695683053&
	41.413458745i\\
	19&12.3280533145& 19.3301891966&
	12.3280533145i& 19.3301891966i\\
	32&11.13352966& 16.85672240&
	17.34237466i\\
	69&11.13352966& 16.85672240&
	17.34237465i\\
	93&8.42836120319& 11.1335296603&
	17.3423746625i\\
	239&13.1823084825& 17.6714531944&
	11.7425210928i\\
	240&3.99263311132& 6.94406875218&
	4.80390756451i& 6.9176905115i\\
	245&3.99263311132& 6.94406875217&
	5.38024923409i& 7.49403218155i
	\\	
\end{tabular}
\end{center}
\medskip

It turns out that we have the following identities, relating unidentified values obtained by means of our method and the generators listed in Table \ref{tab:4}:

\begin{eqnarray*}
  &&\pi^2\times3.78853747194184773010686233752\approx 2\times18.695683053\\
&&  \pi^2\times61.0738884585292464400038231520\approx16\times14.303841078+20\times18.695683053\\
&&  2\pi^2\times0.947134367985461932526715642895\approx18.695683053\\
&&  \pi^2\times3.90285969880676841968003796681\approx 4\times14.303841078- 18.695683053\\
&&  2\pi^{2}\times0.979278824715794481666000593885\approx19.3301891966\\
&&  2\pi^{2}\times4.45674355709313141111341743112\approx2\times12.3280533145-19.3301891966\\
&&  \sqrt2\pi^{2}\times2.99683078705084653614316485878\approx8.42836120319+3\times11.1335296603\\
&&  \sqrt2\pi^{2}\times34.023854315996775681454590398\approx22\times8.42836120319+26\times11.1335296603\\
&&  \sqrt2\pi^2\times2.58823590805561845768157001028\approx-20\times4.80390756451+8\times6.9176905115\\ 
&&  \sqrt2\pi^2\times32.0498374325403392826453731746\approx-168\times4.80390756451+48\times6.9176905115\\
&&  \sqrt2\pi^2\times10.3529436322224738307263862794\approx-80\times4.80390756451+32\times6.9176905115\\
&&  \sqrt2\pi^2\times128.199349730161357130578977414\approx-672\times4.80390756451+192\times 6.9176905115\\
&&  \sqrt2\pi^2\times6.26847094349121003359079492495\approx8\times3.99263311132+8\times 6.94406875217\\ 
&&  \sqrt2\pi^2\times7.96011334055139325749281017005\approx16\times6.94406875217\\
&&\sqrt2\pi^2\times0.839792675513399595961176504110\approx-2\times5.38024923409+3\times 7.49403218155\\
&&\sqrt2\pi^2\times8.12249522907253404678014309610\approx-4\times5.38024923409+18\times 7.49403218155\\
&&9\sqrt2\pi^2\times1.49296475661811062877124579435\approx-32\times5.38024923409+48\times 7.49403218155\\ 
&&9\sqrt2\pi^2\times14.4399915143798181025349622191\approx-64\times5.38024923409+288\times7.49403218155
\end{eqnarray*}

Comparing the first two equalities, we get
\vspace{-3mm}

\begin{eqnarray*}
&&14.303841078\approx\frac{\pi^2}{16}\cdot(61.0738884585292464400038231520-10\cdot3.78853747194184773010686233752).
\end{eqnarray*}

In a similar way we can recover much better approximations of other integrals computed in \cite{periods}. As a surprising consequence we compute not only the period lattice of a rigid Calabi-Yau threefold but also the periods of the singular model $X_{t_0}$ of $\overline{X}_{t_0}$. These additional integrals seem to be crucial for better understanding of the transformation matrices between Frobenius basis at a MUM point and a conifold point.

\section*{Appendix:\\Picard-Fuchs operators of one-parameter families of double octics with a MUM point and a conifold point}
\parindent=0cm
\parskip=2mm
\bgroup
\scriptsize\parskip=3mm
\textbf{2}:
%
\({\Theta}^{4}-1/16\,t \left( 2\,\Theta+1 \right) ^{4}\)

\textbf{8}:
%
\({\Theta}^{4}+1/16\,t \left( 8\,{\Theta}^{2}+8\,\Theta+3 \right)  \left( 2\,\Theta+1 \right) ^{2}+1/16\,{t}^{2} \left( 2\,\Theta+3 \right) ^{2} \left( 2\,\Theta+1 \right) ^{2}\)

%
\(1/4\,\Theta\, \left( \Theta-1 \right)  \left( -1+2\,\Theta \right) ^{2}+1/2\,t{\Theta}^{2} \left( 4\,{\Theta}^{2}+1 \right) +1/16\,{t}^{2} \left( 2\,\Theta+1 \right) ^{4}\)

\textbf{16}:
%
\({\Theta}^{4}+1/4\,t \left( 2\,{\Theta}^{2}+2\,\Theta+1 \right)  \left( 2\,\Theta+1 \right) ^{2}+1/4\,{t}^{2} \left( 2\,\Theta+3 \right)  \left( 2\,\Theta+1 \right)  \left( \Theta+1 \right) ^{2}\)

\textbf{20}:
%
\(1/4\,\Theta\, \left( \Theta-1 \right)  \left( -1+2\,\Theta \right) ^{2}+1/8\,t\Theta\, \left( 12\,{\Theta}^{3}+24\,{\Theta}^{2}-11\,\Theta+3 \right) +{t}^{2} \left( -1/4\,{\Theta}^{4}+11/2\,{\Theta}^{3}+{\frac {79\,{\Theta}^{2}}{16}}+{\frac {17\,\Theta}{16}}+{\frac{7}{32}} \right) +\)

\(+{t}^{3} \left( -{\frac {7\,{\Theta}^{4}}{8}}-5/4\,{\Theta}^{3}+{\frac {119\,{\Theta}^{2}}{32}}+{\frac {45\,\Theta}{32}}+{\frac{5}{16}} \right) +{t}^{4} \left( -3/2\,{\Theta}^{3}-{\frac {9\,{\Theta}^{2}}{8}}-3/8\,\Theta-{\frac{5}{128}} \right) +{\frac {{t}^{5} \left( 2\,\Theta+1 \right) ^{4}}{128}}\)

\textbf{36}:
%
\({\Theta}^{4}+t \left( -2\,{\Theta}^{4}-{\Theta}^{2}-\Theta-1/4 \right) -2\,{t}^{2} \left( 2\,\Theta+1 \right)  \left( {\Theta}^{2}+\Theta+1 \right) +{t}^{3} \left( 2\,{\Theta}^{4}+8\,{\Theta}^{3}+13\,{\Theta}^{2}+9\,\Theta+9/4 \right) -{t}^{4} \left( \Theta+1 \right) ^{4}\)

\textbf{73}:
%
\({\Theta}^{4}+1/4\,t \left( 2\,{\Theta}^{2}+2\,\Theta+1 \right)  \left( 10\,{\Theta}^{2}+10\,\Theta+3 \right) +{t}^{2} \left( 9\,{\Theta}^{4}+36\,{\Theta}^{3}+55\,{\Theta}^{2}+38\,\Theta+21/2 \right)+\)

\(+ {t}^{3} \left( 7\,{\Theta}^{4}+42\,{\Theta}^{3}+90\,{\Theta}^{2}+81\,\Theta+27 \right) +2\,{t}^{4} \left( \Theta+3 \right) ^{2} \left( \Theta+1 \right) ^{2}\)

\textbf{94}:
%
\({\Theta}^{4}+t \left( 11\,{\Theta}^{4}-2\,{\Theta}^{3}+1/4\,{\Theta}^{2}+5/4\,\Theta+{\frac{5}{16}} \right) +{t}^{2} \left( 40\,{\Theta}^{4}+16\,{\Theta}^{3}+{\frac {85\,{\Theta}^{2}}{2}}+14\,\Theta+{\frac{27}{16}} \right) +\)

\(+ {t}^{3} \left( 55\,{\Theta}^{4}+162\,{\Theta}^{3}+149\,{\Theta}^{2}+87\,\Theta+{\frac{85}{4}} \right) +{t}^{4} \left( 29\,{\Theta}^{4}+256\,{\Theta}^{3}+402\,{\Theta}^{2}+268\,\Theta+{\frac{281}{4}} \right) +\)

\(+ {t}^{5} \left( -8\,{\Theta}^{4}+112\,{\Theta}^{3}+397\,{\Theta}^{2}+351\,\Theta+{\frac{437}{4}} \right) +{t}^{6} \left( -20\,{\Theta}^{4}-72\,{\Theta}^{3}+22\,{\Theta}^{2}+96\,\Theta+{\frac{179}{4}} \right) +\)

\(+ {t}^{7} \left( -4\,{\Theta}^{4}-56\,{\Theta}^{3}-112\,{\Theta}^{2}-84\,\Theta-22 \right) +4\,{t}^{8} \left( \Theta+1 \right) ^{4}\)

\textbf{95}:
%
\({\Theta}^{4}+t \left( {\frac {37\,{\Theta}^{4}}{5}}+{\frac {38\,{\Theta}^{3}}{5}}+{\frac {151\,{\Theta}^{2}}{20}}+{\frac {15\,\Theta}{4}}+3/4 \right) +{t}^{2} \left( {\frac {561\,{\Theta}^{4}}{25}}+{\frac {1128\,{\Theta}^{3}}{25}}+{\frac {1406\,{\Theta}^{2}}{25}}+{\frac {361\,\Theta}{10}}+{\frac{763}{80}} \right) +\)

\(+{t}^{3} \left( {\frac {179\,{\Theta}^{4}}{5}}+{\frac {534\,{\Theta}^{3}}{5}}+{\frac {15911\,{\Theta}^{2}}{100}}+{\frac {2361\,\Theta}{20}}+{\frac{699}{20}} \right) +{t}^{4} \left( {\frac {794\,{\Theta}^{4}}{25}}+{\frac {3148\,{\Theta}^{3}}{25}}+{\frac {10879\,{\Theta}^{2}}{50}}+{\frac {8961\,\Theta}{50}}+{\frac{1143}{20}} \right)+\)

\( +{\frac {3\,{t}^{5} \left( \Theta+1 \right)  \left( 124\,{\Theta}^{3}+492\,{\Theta}^{2}+719\,\Theta+366 \right) }{25}}+{\frac {18\,{t}^{6} \left( \Theta+2 \right)  \left( \Theta+1 \right)  \left( 2\,\Theta+3 \right) ^{2}}{25}}\)

\textbf{99}:
%
\({\Theta}^{4}+t \left( 11\,{\Theta}^{4}+10\,{\Theta}^{3}+10\,{\Theta}^{2}+5\,\Theta+1 \right) +{t}^{2} \left( 55\,{\Theta}^{4}+88\,{\Theta}^{3}+108\,{\Theta}^{2}+70\,\Theta+{\frac{295}{16}} \right) +\)

\(+{t}^{3} \left( 165\,{\Theta}^{4}+354\,{\Theta}^{3}+484\,{\Theta}^{2}+351\,\Theta+{\frac{1651}{16}} \right) +{t}^{4} \left( 328\,{\Theta}^{4}+860\,{\Theta}^{3}+1267\,{\Theta}^{2}+954\,\Theta+{\frac{1157}{4}} \right) +\)

\(+{t}^{5} \left( 448\,{\Theta}^{4}+1384\,{\Theta}^{3}+2189\,{\Theta}^{2}+1691\,\Theta+516 \right) +{t}^{6} \left( 420\,{\Theta}^{4}+1512\,{\Theta}^{3}+2564\,{\Theta}^{2}+2070\,\Theta+{\frac{2589}{4}} \right) +\)

\(+{t}^{7} \left( 260\,{\Theta}^{4}+1096\,{\Theta}^{3}+1992\,{\Theta}^{2}+1696\,\Theta+{\frac{2207}{4}} \right) +12\,{t}^{8} \left( \Theta+1 \right)  \left( 8\,{\Theta}^{3}+32\,{\Theta}^{2}+47\,\Theta+24 \right) +\)

\(+4\,{t}^{9} \left( \Theta+2 \right)  \left( \Theta+1 \right)  \left( 2\,\Theta+3 \right) ^{2}\)

\textbf{144}:
%
\({\Theta}^{4}+t \left( 11\,{\Theta}^{4}+10\,{\Theta}^{3}+10\,{\Theta}^{2}+5\,\Theta+1 \right) +{t}^{2} \left( 55\,{\Theta}^{4}+88\,{\Theta}^{3}+108\,{\Theta}^{2}+70\,\Theta+{\frac{295}{16}} \right) +\)

\(+{t}^{3} \left( 165\,{\Theta}^{4}+354\,{\Theta}^{3}+484\,{\Theta}^{2}+351\,\Theta+{\frac{1651}{16}} \right) +{t}^{4} \left( 328\,{\Theta}^{4}+860\,{\Theta}^{3}+1267\,{\Theta}^{2}+954\,\Theta+{\frac{1157}{4}} \right) +\)

\(+{t}^{5} \left( 448\,{\Theta}^{4}+1384\,{\Theta}^{3}+2189\,{\Theta}^{2}+1691\,\Theta+516 \right) +{t}^{6} \left( 420\,{\Theta}^{4}+1512\,{\Theta}^{3}+2564\,{\Theta}^{2}+2070\,\Theta+{\frac{2589}{4}} \right) +\)

\(+{t}^{7} \left( 260\,{\Theta}^{4}+1096\,{\Theta}^{3}+1992\,{\Theta}^{2}+1696\,\Theta+{\frac{2207}{4}} \right) +12\,{t}^{8} \left( \Theta+1 \right)  \left( 8\,{\Theta}^{3}+32\,{\Theta}^{2}+47\,\Theta+24 \right) +\)

\(+4\,{t}^{9} \left( \Theta+2 \right)  \left( \Theta+1 \right)  \left( 2\,\Theta+3 \right) ^{2}\)

\textbf{154}:
%
\(1/4\,\Theta\, \left( \Theta-1 \right)  \left( -1+2\,\Theta \right) ^{2}+1/12\,t\Theta\, \left( 116\,{\Theta}^{3}-96\,{\Theta}^{2}+91\,\Theta-12 \right) +{t}^{2} \left( {\frac {355\,{\Theta}^{4}}{9}}+{\frac {110\,{\Theta}^{3}}{9}}+{\frac {1507\,{\Theta}^{2}}{36}}+{\frac {145\,\Theta}{12}}+3 \right) +\)

\(+{t}^{3} \left( {\frac {793\,{\Theta}^{4}}{9}}+{\frac {1132\,{\Theta}^{3}}{9}}+{\frac {6427\,{\Theta}^{2}}{36}}+{\frac {599\,\Theta}{6}}+{\frac{923}{36}} \right) +{t}^{4} \left( {\frac {1048\,{\Theta}^{4}}{9}}+296\,{\Theta}^{3}+427\,{\Theta}^{2}+294\,\Theta+{\frac{739}{9}} \right) +\)

\(+{t}^{5} \left( {\frac {820\,{\Theta}^{4}}{9}}+{\frac {3008\,{\Theta}^{3}}{9}}+{\frac {4924\,{\Theta}^{2}}{9}}+424\,\Theta+{\frac{1144}{9}} \right) +{\frac {32\,{t}^{6} \left( 11\,{\Theta}^{2}+31\,\Theta+27 \right)  \left( \Theta+1 \right) ^{2}}{9}}+{\frac {64\,{t}^{7} \left( \Theta+2 \right) ^{2} \left( \Theta+1 \right) ^{2}}{9}}\)

\textbf{199}:
%
\(1/4\,\Theta\, \left( \Theta-1 \right)  \left( -1+2\,\Theta \right) ^{2}+1/12\,t\Theta\, \left( 8\,{\Theta}^{3}-24\,{\Theta}^{2}+10\,\Theta-3 \right) -1/36\,{t}^{2} \left( \Theta+1 \right)  \left( 68\,{\Theta}^{3}+92\,{\Theta}^{2}+90\,\Theta+27 \right) +\)

\(+{t}^{3} \left( -4/3\,{\Theta}^{4}-4/3\,{\Theta}^{3}-{\frac {11\,{\Theta}^{2}}{9}}-{\frac {13\,\Theta}{36}}+1/36 \right) +{t}^{4} \left( {\frac {7\,{\Theta}^{4}}{9}}+6\,{\Theta}^{3}+{\frac {49\,{\Theta}^{2}}{4}}+21/2\,\Theta+{\frac{115}{36}} \right) +\)

\(+{t}^{5} \left( 2/3\,{\Theta}^{4}+10/3\,{\Theta}^{3}+11/2\,{\Theta}^{2}+{\frac {23\,\Theta}{6}}+{\frac{35}{36}} \right) +1/9\,{t}^{6} \left( \Theta+1 \right) ^{4}\)

\textbf{242}:
%
\({\Theta}^{4}+t \left( 5\,{\Theta}^{4}+10\,{\Theta}^{3}+10\,{\Theta}^{2}+5\,\Theta+1 \right) +{t}^{2} \left( 9\,{\Theta}^{2}+18\,\Theta+14 \right)  \left( \Theta+1 \right) ^{2}+{t}^{3} \left( \Theta+2 \right)  \left( \Theta+1 \right)  \left( 7\,{\Theta}^{2}+21\,\Theta+18 \right) +\)

\(+2\,{t}^{4} \left( \Theta+3 \right)  \left( \Theta+1 \right)  \left( \Theta+2 \right) ^{2}\)

\textbf{246}:
%
\(1/4\,\Theta\, \left( \Theta-1 \right)  \left( -1+2\,\Theta \right) ^{2}+1/4\,t{\Theta}^{2} \left( 20\,{\Theta}^{2}+11 \right) +{t}^{2} \left( 9\,{\Theta}^{4}+18\,{\Theta}^{3}+22\,{\Theta}^{2}+13\,\Theta+{\frac{51}{16}} \right) +\)

\(+1/16\,{t}^{3} \left( 4\,{\Theta}^{2}+8\,\Theta+7 \right)  \left( 28\,{\Theta}^{2}+56\,\Theta+29 \right) +1/8\,{t}^{4} \left( 2\,\Theta+3 \right) ^{4}\)

\textbf{249}:
%
\({\Theta}^{2} \left( \Theta-1 \right) ^{2}+t{\Theta}^{2} \left( 4\,{\Theta}^{2}+1 \right) +{t}^{2} \left( 5\,{\Theta}^{4}+10\,{\Theta}^{3}+10\,{\Theta}^{2}+5\,\Theta+1 \right) +2\,{t}^{3} \left( \Theta+1 \right) ^{4}\)

\textbf{251}:
%
\(1/4\,\Theta\, \left( \Theta-1 \right)  \left( -1+2\,\Theta \right) ^{2}+1/14\,t\Theta\, \left( 232\,{\Theta}^{3}-360\,{\Theta}^{2}+261\,\Theta-45 \right) +{t}^{2} \left( {\frac {5989\,{\Theta}^{4}}{49}}-{\frac {6526\,{\Theta}^{3}}{49}}+{\frac {12673\,{\Theta}^{2}}{98}}-{\frac {1321\,\Theta}{98}}+{\frac{69}{56}} \right) +\)

\(+{t}^{3} \left( {\frac {25873\,{\Theta}^{4}}{49}}-{\frac {2336\,{\Theta}^{3}}{7}}+{\frac {54315\,{\Theta}^{2}}{98}}+{\frac {115\,\Theta}{49}}+{\frac{11325}{784}} \right) +{t}^{4} \left( {\frac {72358\,{\Theta}^{4}}{49}}-{\frac {14880\,{\Theta}^{3}}{49}}+{\frac {44351\,{\Theta}^{2}}{28}}+{\frac {4785\,\Theta}{28}}+{\frac{45387}{784}} \right) +\)

\(+{t}^{5} \left( {\frac {136135\,{\Theta}^{4}}{49}}+{\frac {2768\,{\Theta}^{3}}{7}}+{\frac {141170\,{\Theta}^{2}}{49}}+{\frac {33581\,\Theta}{98}}+{\frac{19353}{784}} \right) +{t}^{6} \left( {\frac {171792\,{\Theta}^{4}}{49}}+{\frac {52596\,{\Theta}^{3}}{49}}+{\frac {118302\,{\Theta}^{2}}{49}}-{\frac {42657\,\Theta}{49}}-{\frac{4131}{7}} \right) +\)

\(+{t}^{7} \left( {\frac {19472\,{\Theta}^{4}}{7}}-{\frac {1128\,{\Theta}^{3}}{7}}-{\frac {16601\,{\Theta}^{2}}{7}}-{\frac {267159\,\Theta}{49}}-{\frac{470223}{196}} \right) +{t}^{8} \left( {\frac {49752\,{\Theta}^{4}}{49}}-{\frac {169128\,{\Theta}^{3}}{49}}-{\frac {492384\,{\Theta}^{2}}{49}}-{\frac {593742\,\Theta}{49}}-{\frac{471897}{98}} \right) +\)

\(+{t}^{9} \left( -{\frac {19296\,{\Theta}^{4}}{49}}-{\frac {283680\,{\Theta}^{3}}{49}}-{\frac {100188\,{\Theta}^{2}}{7}}-{\frac {747468\,\Theta}{49}}-{\frac{284463}{49}} \right) +{t}^{10} \left( -{\frac {34656\,{\Theta}^{4}}{49}}-{\frac {34560\,{\Theta}^{3}}{7}}-{\frac {563700\,{\Theta}^{2}}{49}}-{\frac {574668\,\Theta}{49}}-{\frac{213219}{49}} \right) +\)

\(+{t}^{11} \left( -{\frac {2736\,{\Theta}^{4}}{7}}-{\frac {117792\,{\Theta}^{3}}{49}}-{\frac {267408\,{\Theta}^{2}}{49}}-{\frac {265968\,\Theta}{49}}-{\frac{96687}{49}} \right) -{\frac {432\,{t}^{12} \left( \Theta+1 \right)  \left( 12\,{\Theta}^{3}+60\,{\Theta}^{2}+101\,\Theta+56 \right) }{49}}-{\frac {144\,{t}^{13} \left( \Theta+2 \right)  \left( \Theta+1 \right)  \left( 2\,\Theta+3 \right) ^{2}}{49}}\)

\textbf{253}:
%
\(1/4\,\Theta\, \left( \Theta-1 \right)  \left( -1+2\,\Theta \right) ^{2}+1/40\,t\Theta\, \left( 492\,{\Theta}^{3}-456\,{\Theta}^{2}+379\,\Theta-57 \right) +{t}^{2} \left( {\frac {6731\,{\Theta}^{4}}{100}}+{\frac {83\,{\Theta}^{3}}{50}}+{\frac {25109\,{\Theta}^{2}}{400}}+{\frac {4989\,\Theta}{400}}+{\frac{553}{160}} \right) +\)

\(+{t}^{3} \left( {\frac {43629\,{\Theta}^{4}}{200}}+{\frac {18723\,{\Theta}^{3}}{100}}+{\frac {264889\,{\Theta}^{2}}{800}}+{\frac {124409\,\Theta}{800}}+{\frac{31623}{800}} \right) +{t}^{4} \left( {\frac {47001\,{\Theta}^{4}}{100}}+{\frac {37083\,{\Theta}^{3}}{50}}+{\frac {880637\,{\Theta}^{2}}{800}}+{\frac {268839\,\Theta}{400}}+{\frac{578301}{3200}} \right) +\)

\(+{t}^{5} \left( {\frac {28631\,{\Theta}^{4}}{40}}+{\frac {156673\,{\Theta}^{3}}{100}}+{\frac {1863431\,{\Theta}^{2}}{800}}+{\frac {1285873\,\Theta}{800}}+{\frac{90707}{200}} \right) +{t}^{6} \left( {\frac {159683\,{\Theta}^{4}}{200}}+{\frac {107837\,{\Theta}^{3}}{50}}+{\frac {2649937\,{\Theta}^{2}}{800}}+{\frac {975887\,\Theta}{400}}+{\frac{568541}{800}} \right) +\)

\(+{t}^{7} \left( {\frac {66527\,{\Theta}^{4}}{100}}+{\frac {52122\,{\Theta}^{3}}{25}}+{\frac {66313\,{\Theta}^{2}}{20}}+{\frac {505161\,\Theta}{200}}+{\frac{149699}{200}} \right) +{t}^{8} \left( {\frac {83369\,{\Theta}^{4}}{200}}+{\frac {36614\,{\Theta}^{3}}{25}}+{\frac {482423\,{\Theta}^{2}}{200}}+{\frac {187197\,\Theta}{100}}+{\frac{22373}{40}} \right) +\)

\(+{t}^{9} \left( {\frac {19509\,{\Theta}^{4}}{100}}+{\frac {37767\,{\Theta}^{3}}{50}}+{\frac {129321\,{\Theta}^{2}}{100}}+{\frac {51153\,\Theta}{50}}+{\frac{7713}{25}} \right) +{t}^{10} \left( {\frac {333\,{\Theta}^{4}}{5}}+{\frac {7068\,{\Theta}^{3}}{25}}+{\frac {25289\,{\Theta}^{2}}{50}}+{\frac {10299\,\Theta}{25}}+{\frac{6327}{50}} \right) +\)

\(+{t}^{11} \left( {\frac {394\,{\Theta}^{4}}{25}}+{\frac {1844\,{\Theta}^{3}}{25}}+{\frac {3464\,{\Theta}^{2}}{25}}+{\frac {2926\,\Theta}{25}}+{\frac{922}{25}} \right) +{\frac{2}{25}}\,{t}^{12} \left( 29\,{\Theta}^{2}+94\,\Theta+86 \right)  \left( \Theta+1 \right) ^{2}+{\frac {4\,{t}^{13} \left( \Theta+2 \right) ^{2} \left( \Theta+1 \right) ^{2}}{25}}\)

\textbf{254}:
%
\({\Theta}^{4}+t \left( 10\,{\Theta}^{4}+14\,{\Theta}^{3}+{\frac{25}{2}}\,{\Theta}^{2}+11/2\,\Theta+{\frac{17}{16}} \right) +{t}^{2} \left( 62\,{\Theta}^{4}+74\,{\Theta}^{3}+104\,{\Theta}^{2}+{\frac {171\,\Theta}{2}}+{\frac{399}{16}} \right) +\)

\(+{t}^{3} \left( 293\,{\Theta}^{4}+276\,{\Theta}^{3}+{\frac {271\,{\Theta}^{2}}{2}}+228\,\Theta+{\frac{331}{4}} \right) +{t}^{4} \left( 1039\,{\Theta}^{4}+944\,{\Theta}^{3}+338\,{\Theta}^{2}-47\,\Theta-{\frac{1829}{16}} \right) +\)

\(+{t}^{5} \left( 2884\,{\Theta}^{4}+2098\,{\Theta}^{3}+{\frac {5937\,{\Theta}^{2}}{2}}+{\frac {275\,\Theta}{2}}-{\frac{4281}{8}} \right) +{t}^{6} \left( 6348\,{\Theta}^{4}+4278\,{\Theta}^{3}+9109\,{\Theta}^{2}+{\frac {6591\,\Theta}{2}}+72 \right) +\)

\(+{t}^{7} \left( 10465\,{\Theta}^{4}+10988\,{\Theta}^{3}+{\frac {36131\,{\Theta}^{2}}{2}}+10011\,\Theta+{\frac{37033}{16}} \right) +{t}^{8} \left( 12062\,{\Theta}^{4}+21584\,{\Theta}^{3}+30237\,{\Theta}^{2}+19844\,\Theta+{\frac{44823}{8}} \right) +\)

\(+{t}^{9} \left( 9196\,{\Theta}^{4}+25056\,{\Theta}^{3}+37002\,{\Theta}^{2}+27144\,\Theta+{\frac{32891}{4}} \right) +{t}^{10} \left( 4360\,{\Theta}^{4}+16352\,{\Theta}^{3}+27484\,{\Theta}^{2}+22392\,\Theta+{\frac{14529}{2}} \right) +\)

\(+{t}^{11} \left( 1152\,{\Theta}^{4}+5568\,{\Theta}^{3}+10816\,{\Theta}^{2}+9776\,\Theta+3416 \right) +8\,{t}^{12} \left( 2\,\Theta+3 \right) ^{4}\)

\textbf{255}:
\({\Theta}^{4}+1/4\,t\Theta\, \left( 4\,{\Theta}^{3}-10\,{\Theta}^{2}-6\,\Theta-1 \right) +{t}^{2} \left( -{\frac {23\,{\Theta}^{4}}{16}}-{\frac {31\,{\Theta}^{3}}{8}}-{\frac {25\,{\Theta}^{2}}{8}}-{\frac {57\,\Theta}{16}}-{\frac{89}{64}} \right) +\)

\(+{t}^{3} \left( -{\frac {21\,{\Theta}^{4}}{32}}-{\frac {21\,{\Theta}^{3}}{8}}-{\frac {83\,{\Theta}^{2}}{16}}-{\frac {33\,\Theta}{32}}+{\frac{51}{128}} \right) +{t}^{4} \left( {\frac {59\,{\Theta}^{4}}{64}}+{\frac {41\,{\Theta}^{3}}{32}}+{\frac {279\,{\Theta}^{2}}{32}}+{\frac {923\,\Theta}{128}}+{\frac{2567}{1024}} \right) +\)

\(+{t}^{5} \left( -{\frac {11\,{\Theta}^{4}}{32}}+{\frac {91\,{\Theta}^{3}}{32}}+{\frac {3477\,{\Theta}^{2}}{512}}+{\frac {1907\,\Theta}{256}}+{\frac{1741}{512}} \right) +{t}^{6} \left( -{\frac {27\,{\Theta}^{4}}{128}}-{\frac {75\,{\Theta}^{3}}{64}}-{\frac {9\,{\Theta}^{2}}{16}}-{\frac {657\,\Theta}{512}}-{\frac{697}{1024}} \right) +\)

\(+{t}^{7} \left( {\frac {57\,{\Theta}^{4}}{256}}+{\frac {39\,{\Theta}^{3}}{128}}-{\frac {3085\,{\Theta}^{2}}{2048}}-{\frac {2795\,\Theta}{1024}}-{\frac{3089}{2048}} \right)+{t}^{8} \left( -{\frac {11\,{\Theta}^{4}}{512}}+{\frac {113\,{\Theta}^{3}}{256}}+{\frac {601\,{\Theta}^{2}}{2048}}+{\frac {9\,\Theta}{256}}-{\frac{1027}{8192}} \right) +\)

\(+{t}^{9} \left( -{\frac {13\,{\Theta}^{4}}{512}}-{\frac {57\,{\Theta}^{3}}{256}}-{\frac {2581\,{\Theta}^{2}}{8192}}-{\frac {717\,\Theta}{4096}}-{\frac{73}{8192}} \right) +{t}^{10} \left( {\frac {61\,{\Theta}^{4}}{4096}}-{\frac {11\,{\Theta}^{3}}{2048}}-{\frac {359\,{\Theta}^{2}}{4096}}-{\frac {1005\,\Theta}{8192}}-{\frac{205}{4096}} \right) +\)

\(+{t}^{11} \left( -{\frac {5\,{\Theta}^{4}}{8192}}+{\frac {5\,{\Theta}^{3}}{1024}}+{\frac {725\,{\Theta}^{2}}{32768}}+{\frac {475\,\Theta}{16384}}+{\frac{205}{16384}} \right) -{\frac {25\,{t}^{12} \left( 2\,\Theta+3 \right) ^{4}}{262144}}\)

\textbf{256}:
%
\(1/16\, \left( 2\,\Theta+1 \right) ^{4}+t \left( -21/2\,{\Theta}^{4}-12\,{\Theta}^{3}+3\,{\Theta}^{2}+15/2\,\Theta+{\frac{35}{16}} \right) +{t}^{2} \left( {\frac {61\,{\Theta}^{4}}{2}}+15\,{\Theta}^{3}-{\frac {21\,{\Theta}^{2}}{4}}+3\,\Theta+{\frac{95}{32}} \right) +\)

\(+{t}^{3} \left( -{\frac {121\,{\Theta}^{4}}{8}}-22\,{\Theta}^{3}-{\frac {51\,{\Theta}^{2}}{4}}-{\frac {23\,\Theta}{8}}+{\frac{23}{64}} \right) -3/16\,{t}^{4} \left( \Theta+1 \right)  \left( 44\,{\Theta}^{2}+88\,\Theta+61 \right) +\)

\(+{t}^{5} \left( {\frac {121\,{\Theta}^{4}}{32}}+{\frac {99\,{\Theta}^{3}}{4}}+{\frac {975\,{\Theta}^{2}}{16}}+{\frac {2145\,\Theta}{32}}+{\frac{7097}{256}} \right) +{t}^{6} \left( -{\frac {61\,{\Theta}^{4}}{32}}-{\frac {229\,{\Theta}^{3}}{16}}-{\frac {2547\,{\Theta}^{2}}{64}}-{\frac {193\,\Theta}{4}}-{\frac{11007}{512}} \right) +\)

\(+{t}^{7} \left( {\frac {21\,{\Theta}^{4}}{128}}+{\frac {9\,{\Theta}^{3}}{8}}+{\frac {177\,{\Theta}^{2}}{64}}+{\frac {375\,\Theta}{128}}+{\frac{1165}{1024}} \right) -{\frac {{t}^{8} \left( 2\,\Theta+3 \right) ^{4}}{4096}}\)

\textbf{257}:
%
\({\Theta}^{2} \left( \Theta-1 \right) ^{2}-1/4\,t{\Theta}^{2} \left( 13\,{\Theta}^{2}-12\,\Theta+9 \right) +{t}^{2} \left( {\frac {33\,{\Theta}^{4}}{8}}+{\frac {15\,{\Theta}^{3}}{4}}+{\frac {17\,{\Theta}^{2}}{4}}+{\frac {13\,\Theta}{8}}+3/8 \right) +\)

\({t}^{3} \left( -{\frac {77\,{\Theta}^{4}}{32}}-{\frac {43\,{\Theta}^{3}}{4}}-{\frac {677\,{\Theta}^{2}}{64}}-{\frac {119\,\Theta}{16}}-{\frac{33}{16}} \right) +{t}^{4} \left( {\frac {85\,{\Theta}^{4}}{256}}+{\frac {1059\,{\Theta}^{3}}{128}}+{\frac {3735\,{\Theta}^{2}}{256}}+{\frac {753\,\Theta}{64}}+{\frac{499}{128}} \right) +\)

\(+{t}^{5} \left( {\frac {423\,{\Theta}^{4}}{1024}}-{\frac {495\,{\Theta}^{3}}{256}}-{\frac {9145\,{\Theta}^{2}}{1024}}-{\frac {2317\,\Theta}{256}}-{\frac{3715}{1024}} \right) +{t}^{6} \left( -{\frac {271\,{\Theta}^{4}}{1024}}-{\frac {419\,{\Theta}^{3}}{512}}+{\frac {1699\,{\Theta}^{2}}{1024}}+{\frac {6125\,\Theta}{2048}}+{\frac{3229}{2048}} \right) +\)

\(+{t}^{7} \left( {\frac {187\,{\Theta}^{4}}{4096}}+{\frac {297\,{\Theta}^{3}}{512}}+{\frac {9423\,{\Theta}^{2}}{16384}}+{\frac {105\,\Theta}{2048}}-{\frac{695}{4096}} \right)+{t}^{8} \left( {\frac {111\,{\Theta}^{4}}{8192}}-{\frac {321\,{\Theta}^{3}}{4096}}-{\frac {1859\,{\Theta}^{2}}{8192}}-{\frac {1745\,\Theta}{8192}}-{\frac{1055}{16384}} \right) +\)

\(+{t}^{9} \left( -{\frac {55\,{\Theta}^{4}}{8192}}-{\frac {5\,{\Theta}^{3}}{256}}-{\frac {1125\,{\Theta}^{2}}{65536}}-{\frac {25\,\Theta}{32768}}+{\frac{955}{262144}} \right) +{\frac {25\,{t}^{10} \left( 2\,\Theta+3 \right) ^{4}}{524288}}\)

\textbf{259}:
%
\(1/4\,\Theta\, \left( \Theta-1 \right)  \left( -1+2\,\Theta \right) ^{2}+1/2\,t\Theta\, \left( 8\,{\Theta}^{2}-4\,\Theta+1 \right) +{t}^{2} \left( -3\,{\Theta}^{4}-2\,{\Theta}^{3}-{\frac {15\,{\Theta}^{2}}{4}}-{\frac {17\,\Theta}{4}}-1 \right) -\)

\(-1/2\,{t}^{3} \left( 2\,\Theta+1 \right)  \left( 8\,{\Theta}^{2}+8\,\Theta+13 \right) +{t}^{4} \left( 3\,{\Theta}^{4}+10\,{\Theta}^{3}+{\frac {63\,{\Theta}^{2}}{4}}+{\frac {37\,\Theta}{4}}+3/2 \right) +1/2\,{t}^{5} \left( \Theta+1 \right)  \left( 8\,{\Theta}^{2}+20\,\Theta+13 \right) -\)

\(-1/4\,{t}^{6} \left( \Theta+2 \right)  \left( \Theta+1 \right)  \left( 2\,\Theta+3 \right) ^{2}\)

\textbf{262}:
%
\(1/4\,\Theta\, \left( \Theta-1 \right)  \left( -1+2\,\Theta \right) ^{2}+1/28\,t\Theta\, \left( 164\,{\Theta}^{3}-72\,{\Theta}^{2}+111\,\Theta-9 \right) +{t}^{2} \left( {\frac {807\,{\Theta}^{4}}{49}}+{\frac {930\,{\Theta}^{3}}{49}}+{\frac {2623\,{\Theta}^{2}}{98}}+{\frac {1261\,\Theta}{98}}+{\frac{177}{56}} \right) +\)

\(+{t}^{3} \left( {\frac {1417\,{\Theta}^{4}}{49}}+{\frac {4012\,{\Theta}^{3}}{49}}+{\frac {23965\,{\Theta}^{2}}{196}}+{\frac {17571\,\Theta}{196}}+{\frac{22095}{784}} \right) +{t}^{4} \left( {\frac {1670\,{\Theta}^{4}}{49}}+{\frac {7806\,{\Theta}^{3}}{49}}+{\frac {58725\,{\Theta}^{2}}{196}}+{\frac {52533\,\Theta}{196}}+{\frac{76995}{784}} \right) +\)

\(+{t}^{5} \left( {\frac {1308\,{\Theta}^{4}}{49}}+{\frac {9096\,{\Theta}^{3}}{49}}+{\frac {86145\,{\Theta}^{2}}{196}}+{\frac {89409\,\Theta}{196}}+{\frac{36555}{196}} \right) +{t}^{6} \left( 12\,{\Theta}^{4}+{\frac {930\,{\Theta}^{3}}{7}}+{\frac {79011\,{\Theta}^{2}}{196}}+{\frac {94125\,\Theta}{196}}+{\frac{24207}{112}} \right) +\)

\(+{t}^{7} \left( {\frac {12\,{\Theta}^{4}}{49}}+{\frac {2340\,{\Theta}^{3}}{49}}+{\frac {10656\,{\Theta}^{2}}{49}}+{\frac {29817\,\Theta}{98}}+{\frac{2115}{14}} \right) +{t}^{8} \left( -{\frac {177\,{\Theta}^{4}}{49}}-{\frac {282\,{\Theta}^{3}}{49}}+{\frac {9369\,{\Theta}^{2}}{196}}+{\frac {19419\,\Theta}{196}}+{\frac{22899}{392}} \right) +\)

\(+{t}^{9} \left( -{\frac {115\,{\Theta}^{4}}{49}}-{\frac {724\,{\Theta}^{3}}{49}}-{\frac {1495\,{\Theta}^{2}}{98}}+{\frac {151\,\Theta}{49}}+{\frac{3639}{392}} \right) +{t}^{10} \left( -{\frac {26\,{\Theta}^{4}}{49}}-{\frac {306\,{\Theta}^{3}}{49}}-{\frac {1275\,{\Theta}^{2}}{98}}-{\frac {999\,\Theta}{98}}-{\frac{1683}{784}} \right) +\)

\(+{t}^{11} \left( {\frac {6\,{\Theta}^{4}}{49}}-{\frac {24\,{\Theta}^{3}}{49}}-{\frac {401\,{\Theta}^{2}}{196}}-{\frac {461\,\Theta}{196}}-{\frac{663}{784}} \right) +{t}^{12} \left( {\frac {5\,{\Theta}^{4}}{49}}+{\frac {20\,{\Theta}^{3}}{49}}+{\frac {32\,{\Theta}^{2}}{49}}+{\frac {24\,\Theta}{49}}+{\frac{57}{392}} \right) +{\frac {{t}^{13} \left( 2\,\Theta+3 \right) ^{4}}{784}}\)

\textbf{265}:
%
\(1/4\,{\Theta}^{2} \left( -1+2\,\Theta \right)  \left( 2\,\Theta+1 \right) -1/32\,t \left( 2\,\Theta+1 \right)  \left( 32\,{\Theta}^{3}+16\,{\Theta}^{2}+22\,\Theta+7 \right) +{t}^{2} \left( {\frac {25\,{\Theta}^{4}}{16}}+{\frac {13\,{\Theta}^{3}}{4}}+{\frac {129\,{\Theta}^{2}}{32}}+{\frac {41\,\Theta}{16}}+{\frac{165}{256}} \right) +\)

\(+{t}^{3} \left( -{\frac {19\,{\Theta}^{4}}{32}}-{\frac {31\,{\Theta}^{3}}{16}}-{\frac {189\,{\Theta}^{2}}{64}}-{\frac {143\,\Theta}{64}}-{\frac{347}{512}} \right) +{t}^{4} \left( {\frac {7\,{\Theta}^{4}}{64}}+1/2\,{\Theta}^{3}+{\frac {119\,{\Theta}^{2}}{128}}+{\frac {13\,\Theta}{16}}+{\frac{283}{1024}} \right) -{\frac {{t}^{5} \left( 2\,\Theta+3 \right) ^{4}}{2048}}\)

\textbf{268}:
%
\({\Theta}^{4}-t\Theta\, \left( 4\,{\Theta}^{3}-10\,{\Theta}^{2}-6\,\Theta-1 \right) +{t}^{2} \left( -23\,{\Theta}^{4}-62\,{\Theta}^{3}-50\,{\Theta}^{2}-57\,\Theta-{\frac{89}{4}} \right) +{t}^{3} \left( 42\,{\Theta}^{4}+168\,{\Theta}^{3}+332\,{\Theta}^{2}+66\,\Theta-{\frac{51}{2}} \right) +\)

\(+{t}^{4} \left( 236\,{\Theta}^{4}+328\,{\Theta}^{3}+2232\,{\Theta}^{2}+1846\,\Theta+{\frac{2567}{4}} \right) +{t}^{5} \left( 352\,{\Theta}^{4}-2912\,{\Theta}^{3}-6954\,{\Theta}^{2}-7628\,\Theta-3482 \right) +\)

\(+{t}^{6} \left( -864\,{\Theta}^{4}-4800\,{\Theta}^{3}-2304\,{\Theta}^{2}-5256\,\Theta-2788 \right) +{t}^{7} \left( -3648\,{\Theta}^{4}-4992\,{\Theta}^{3}+24680\,{\Theta}^{2}+44720\,\Theta+24712 \right) +\)

\(+{t}^{8} \left( -1408\,{\Theta}^{4}+28928\,{\Theta}^{3}+19232\,{\Theta}^{2}+2304\,\Theta-8216 \right) +{t}^{9} \left( 6656\,{\Theta}^{4}+58368\,{\Theta}^{3}+82592\,{\Theta}^{2}+45888\,\Theta+2336 \right) +\)

\(+{t}^{10} \left( 15616\,{\Theta}^{4}-5632\,{\Theta}^{3}-91904\,{\Theta}^{2}-128640\,\Theta-52480 \right) +{t}^{11} \left( 2560\,{\Theta}^{4}-20480\,{\Theta}^{3}-92800\,{\Theta}^{2}-121600\,\Theta-52480 \right) -\)

\(-1600\,{t}^{12} \left( 2\,\Theta+3 \right) ^{4}\)

\textbf{274}:
%
\(1/4\,\Theta\, \left( \Theta-1 \right)  \left( -1+2\,\Theta \right) ^{2}+1/40\,t\Theta\, \left( 332\,{\Theta}^{3}-216\,{\Theta}^{2}+219\,\Theta-27 \right) +{t}^{2} \left( {\frac {2871\,{\Theta}^{4}}{100}}+{\frac {1413\,{\Theta}^{3}}{50}}+{\frac {12849\,{\Theta}^{2}}{400}}+{\frac {5199\,\Theta}{400}}+{\frac{95}{32}} \right) +\)

\(+{t}^{3} \left( {\frac {10169\,{\Theta}^{4}}{200}}+{\frac {17737\,{\Theta}^{3}}{100}}+{\frac {152361\,{\Theta}^{2}}{800}}+{\frac {92291\,\Theta}{800}}+{\frac{999}{32}} \right) +{t}^{4} \left( {\frac {1873\,{\Theta}^{4}}{50}}+{\frac {40143\,{\Theta}^{3}}{100}}+{\frac {506483\,{\Theta}^{2}}{800}}+{\frac {352509\,\Theta}{800}}+{\frac{89327}{640}} \right) +\)

\(+{t}^{5} \left( -{\frac {2871\,{\Theta}^{4}}{100}}+{\frac {42237\,{\Theta}^{3}}{100}}+{\frac {942333\,{\Theta}^{2}}{800}}+{\frac {752883\,\Theta}{800}}+{\frac{545069}{1600}} \right) +{t}^{6} \left( -{\frac {19521\,{\Theta}^{4}}{200}}+{\frac {4491\,{\Theta}^{3}}{100}}+{\frac {962217\,{\Theta}^{2}}{800}}+{\frac {185391\,\Theta}{160}}+{\frac{1549843}{3200}} \right) +\)

\(+{t}^{7} \left( -{\frac {19521\,{\Theta}^{4}}{200}}-{\frac {43533\,{\Theta}^{3}}{100}}+{\frac {385929\,{\Theta}^{2}}{800}}+{\frac {577359\,\Theta}{800}}+{\frac{1234843}{3200}} \right) +{t}^{8} \left( -{\frac {2871\,{\Theta}^{4}}{100}}-{\frac {53721\,{\Theta}^{3}}{100}}-{\frac {209163\,{\Theta}^{2}}{800}}+{\frac {26223\,\Theta}{800}}+{\frac{202241}{1600}} \right) +\)

\(+{t}^{9} \left( {\frac {1873\,{\Theta}^{4}}{50}}-{\frac {25159\,{\Theta}^{3}}{100}}-{\frac {277141\,{\Theta}^{2}}{800}}-{\frac {183103\,\Theta}{800}}-{\frac{102173}{3200}} \right) +{t}^{10} \left( {\frac {10169\,{\Theta}^{4}}{200}}+{\frac {2601\,{\Theta}^{3}}{100}}-{\frac {29271\,{\Theta}^{2}}{800}}-{\frac {50553\,\Theta}{800}}-{\frac{647}{32}} \right) +\)

\(+{t}^{11} \left( {\frac {2871\,{\Theta}^{4}}{100}}+{\frac {4329\,{\Theta}^{3}}{50}}+{\frac {47841\,{\Theta}^{2}}{400}}+{\frac {32523\,\Theta}{400}}+{\frac{3607}{160}} \right) +1/40\,{t}^{12} \left( \Theta+1 \right)  \left( 332\,{\Theta}^{3}+1212\,{\Theta}^{2}+1647\,\Theta+794 \right) +\)

\(+1/4\,{t}^{13} \left( \Theta+2 \right)  \left( \Theta+1 \right)  \left( 2\,\Theta+3 \right) ^{2}\)


\egroup

\end{document}